\documentclass[12pt, reqno]{amsart}
\usepackage{amsmath, amsthm, amscd, amsfonts, amssymb, graphicx, color}
\usepackage[bookmarksnumbered, colorlinks, plainpages]{hyperref}
\input{mathrsfs.sty}

\newtheorem{theorem}{Theorem}[section]
\newtheorem{lemma}[theorem]{Lemma}

\theoremstyle{definition}

\theoremstyle{remark}
\newtheorem{remark}[theorem]{Remark}

\numberwithin{equation}{section}

\begin{document}

\title{Operator revision of a Ky Fan type inequality}

\author[J. Rooin, A. Alikhani]{J. Rooin$^1$, A. Alikhani$^1$ }

\address{$^{1}$Department of Mathematics, Institute for Advanced Studies in Basic Sciences (IASBS), Zanjan 45137-66731, Iran}

\email{rooin@iasbs.ac.ir\\ akram.alikhani88@gmail.com}



\begin{abstract}
Let $\mathscr{H}$ be a complex Hilbert space and $A,B\in \mathbb{B}(\mathscr{H})$
such that $0<A,B\leq\frac{1}{2}I$. Setting $A':=I-A$ and $B':=I-B$, we prove
$$
A'\nabla_\lambda B'-A'!_\lambda B'
\leq A\nabla_\lambda B-A!_\lambda B,
$$
where $\nabla_\lambda$ and $!_\lambda$ denote
the weighted arithmetic and
harmonic operator means, respectively. This inequality is the natural extension of
a Ky Fan type inequality due to H. Alzer.
Some parallel and related results  are also obtained.
\end{abstract}
\maketitle

\section{Introduction}
Let $n\geq2$ and $\lambda_1,\cdots,\lambda_n\geq0$
such that $\sum_{i=1}^n\lambda_i=1$.
For $n$ arbitrary real numbers $x_1,\cdots,x_n>0$,
 we denote by $A_n, G_n$ and $H_n$ the arithmetic,
geometric and harmonic means of $x_1,\cdots,x_n$ respectively, i.e.
\begin{equation}\label{numer-1-mean}
A_n=\sum_{i=1}^n\lambda_ix_i\qquad\qquad G_n=\prod_{i=1}^n{x_i}^{\lambda_i}\qquad\qquad
H_n=\frac{1}{\sum_{i=1}^n\lambda_i\frac{1}{x_i}}.
\end{equation}
Also when $x_i\in(0,\frac{1}{2}]$ we denote by
 $A'_n, G'_n$ and $H'_n$ the arithmetic, geometric
and harmonic means of $x'_1:=1-x_1,\cdots,x'_n:=1-x_n$ respectively, i.e.
\begin{equation}\label{numer-2-mean}
A'_n=\sum_{i=1}^n\lambda_ix'_i\qquad\qquad G'_n=\prod_{i=1}^n{x'_i}^{\lambda_i}\qquad\qquad
H'_n=\frac{1}{\sum_{i=1}^n\lambda_i\frac{1}{x'_i}}.
\end{equation}
The Ky Fan's inequality
\begin{equation}\label{ky-fan-1}
\frac{A'_n}{G'_n}\leq\frac{A_n}{G_n}
\end{equation}
was introduced for the first time in \cite{B&B}.
From then several mathematicians attained new proofs,
extensions, refinements and various related results.
For more information about Ky Fan and Ky Fan type inequalities
 see \cite{Alzer-T,MO}.
In 1988, an additive analogue of (\ref{ky-fan-1})
presented by H. Alzer \cite{Alzer-U} as
\begin{equation}\label{ky-fan-2}
A'_n-G'_n\leq A_n-G_n.
\end{equation}
In both of (\ref{ky-fan-1}) and (\ref{ky-fan-2}),
equality holds if and only if $x_1=\cdots=x_n$.\\
Another interesting additive analogue of Ky Fan's
inequality was discovered by H. Alzer
\cite{Alzer-A} in 1993, as follows
\begin{equation}\label{ky-fan-3}
A'_n-H'_n\leq A_n-H_n
\end{equation}
with equality holding if and only if $x_1=\cdots=x_n$. For a new refinement and
converse of (\ref{ky-fan-3}) refer to \cite{JR}.
If we divide two sides of (\ref{ky-fan-3}) by $A'_nH'_n$, and using this fact that
$A'_nH'_n\geq A_nH_n$, we conclude the following inequality due to J. Sandor \cite{sandor}
\begin{equation}\label{ky-fan-4}
\frac{1}{H'_n}-\frac{1}{A'_n}\leq\frac{1}{H_n}-\frac{1}{A_n}.
\end{equation}
Similarly, dividing two sides of (\ref{ky-fan-3}) by $H'_n$ and note that
$H'_n\geq H_n$,
we obtain
\begin{equation}\label{ky-fan-5}
\frac{A'_n}{H'_n}\leq \frac{A_n}{H_n}.
\end{equation}
Throughout this paper, let $\mathbb{B}(\mathscr{H})$ denote the
$C^*-$algebra of all bounded linear operators acting on a complex
Hilbert space $(\mathscr{H} ,\langle \cdot,\cdot\rangle)$ and $I$ be
the identity operator.  An operator
$A\in\mathbb{B}(\mathscr{H})$ is called positive if $\langle
Ax,x\rangle\geq0$ holds for every $x\in\mathscr{H}$ and then we
write $A\geq0$. If $A$ is positive and invertible, we write $A>0$.
For self-adjoint operators $A, B$, we say $A\leq B$ if $B-A\geq0$.
If $A\geq0$ and $X\in \mathbb{B}(\mathscr{H})$, then $X^*AX\geq0$.
We define the weighted arithmetic mean $\nabla_\lambda$,
the weighted geometric mean (the $\lambda$-power mean) $\sharp_\lambda$
 and the weighted harmonic mean $!_\lambda$
for $A,B>0$ and $0\leq\lambda\leq1$ by
\begin{eqnarray*}
&& A\nabla_\lambda B:=(1-\lambda)A+\lambda B\\
&& A\sharp_\lambda B:= A^{\frac{1}{2}}\big(A^{-\frac{1}{2}}
B^{\frac{1}{2}}A^{-\frac{1}{2}}\big)^{\lambda}A^{\frac{1}{2}}\\
&& A!_\lambda B:=\Big((1-\lambda)A^{-1}+\lambda B^{-1}\Big)^{-1}.
\end{eqnarray*}
In particular, in the case of $\lambda=\frac{1}{2}$, the usual arithmetic,
geometric and harmonic means of $A, B$ simply denoted by
$A\nabla B$, $A\sharp B$ and $A!B$, respectively.
For more information
on operator inequalities and operator means see \cite{FMPS}.\\
In \cite{MO-MI} an operator revision of (\ref{ky-fan-2}) for $n=2$ in the
case of commutative operators are obtained.
The aim of this paper is to generalize the inequalities (\ref{ky-fan-3}),
(\ref{ky-fan-4}) and (\ref{ky-fan-5})
for operators on Hilbert spaces.
In this way, we obtain natural direct operator version
for the inequality (\ref{ky-fan-3}).
Also we give operator extensions of
(\ref{ky-fan-4}) and (\ref{ky-fan-5}).
\section{Main results}
We start with the following useful lemma.
\begin{lemma}
Let $T\in \mathbb{B}(\mathscr{H})$ be a strictly positive operator
and $\lambda\in[0,1]$. Then
\begin{align}
\label{lem-f}&(i)\quad I\nabla_\lambda T-I!_\lambda T=\lambda(1-\lambda)(I-T)(T\nabla_\lambda I)^{-1}(I-T).\\
\label{lem-f-s}&(ii)\quad T^{\frac{1}{2}}\big(I\nabla_\lambda T^{-1}-I!_\lambda T^{-1}\big)T^{\frac{1}{2}}
=\lambda(1-\lambda)(T-I)(I\nabla_\lambda T)^{-1}(T-I).\\
\label{lem-f-ex}&(iii)\quad(I!_\lambda T)^{-\frac{1}{2}}(I\nabla_\lambda T)(I!_\lambda T)^{-\frac{1}{2}}-I=
\lambda(1-\lambda)(I-T)T^{-1}(I-T)
\end{align}
\end{lemma}
\begin{proof}
(i) Since the operators $I$ and $T$ commute, we have
\begin{eqnarray*}
I\nabla_\lambda T-I!_\lambda T
&=&(1-\lambda)I+\lambda T-\big((1-\lambda)I+\lambda T^{-1}\big)^{-1}\\
&=&(1-\lambda)I+\lambda T-\big((1-\lambda)T+\lambda I\big)^{-1}T\\
&=&\big((1-\lambda)T+\lambda I\big)^{-1}\Big[\big((1-\lambda)T+\lambda I\big)\big((1-\lambda)I+\lambda T\big)-T\Big]\\
&=&\lambda(1-\lambda)\big((1-\lambda)T+\lambda I\big)^{-1}(I-2T+T^2)\\
&=&\lambda(1-\lambda)(I-T)\big((1-\lambda)T+\lambda I\big)^{-1}(I-T)\\
&=&\lambda(1-\lambda)(I-T)(T\nabla_\lambda I)^{-1}(I-T).
\end{eqnarray*}
\noindent (ii) Changing $T$ by $T^{-1}$ in (\ref{lem-f}), we have
\begin{eqnarray}
\nonumber I\nabla_\lambda T^{-1}-I!_\lambda T^{-1}&=&
\lambda(1-\lambda)(I-T^{-1})(T^{-1}\nabla_\lambda I)^{-1}(I-T^{-1})\\
\label{p-l-(ii)}&=&\lambda(1-\lambda)(T-I)T^{-1}(I\nabla_\lambda T)^{-1}(T-I).
\end{eqnarray}
Now multiplying both sides of (\ref{p-l-(ii)}) from left and right by $T^{\frac{1}{2}}$, we deduce
(\ref{lem-f-s}). \\
(iii) First, we multiply both sides of (\ref{lem-f}) from left and right by
$(I!_\lambda T)^{-\frac{1}{2}}$. Since
$$
(I!_\lambda T)^{-1}(T\nabla_\lambda I)^{-1}=\big((1-\lambda)I+\lambda T^{-1}\big)
\big((1-\lambda)T+\lambda I\big)^{-1}=T^{-1},
$$
so we obtain (iii).
\end{proof}
\begin{theorem}
Let $A,B$ be two strictly positive operators
and $\lambda\in[0,1]$. Then
{\small\begin{align}
\label{th-ex-lem-f}&(i)~~~A\nabla_\lambda B-A!_\lambda B=\lambda(1-\lambda)(A-B)(B\nabla_\lambda A)^{-1}(A-B).\\
\label{th-ex-lem-f-s}&(ii)~~~(A\sharp B)\Big[(A!_\lambda B)^{-1}-(A\nabla_\lambda B)^{-1}\Big]
(A\sharp B)=\lambda(1-\lambda)(B-A)(A\nabla_\lambda B)^{-1}(B-A).\\
\nonumber &(iii)~~~A\Big(A^{-1}\sharp(A!_\lambda B)^{-1}\Big)(A\nabla_\lambda B)
\Big(A^{-1}\sharp(A!_\lambda B)^{-1}\Big)A-A\\
\label{th-ex-lem-f-ex}&\qquad=\lambda(1-\lambda)(A-B)B^{-1}(A-B).
\end{align}
}
\end{theorem}
\begin{proof}
(i) Considering strictly positive operator $T=A^{-\frac{1}{2}}BA^{-\frac{1}{2}}$ and putting it
in (\ref{lem-f}) we obtain
{\small\begin{align}
\nonumber&\Big((1-\lambda)I+\lambda A^{-\frac{1}{2}}BA^{-\frac{1}{2}}\Big)
-\Big((1-\lambda)I+\lambda A^{\frac{1}{2}}B^{-1}A^{\frac{1}{2}}\Big)^{-1}\\
\label{eq-1} &=\lambda(1-\lambda)(I-A^{-\frac{1}{2}}BA^{-\frac{1}{2}})
\Big((1-\lambda)A^{-\frac{1}{2}}BA^{-\frac{1}{2}}+\lambda I\Big)^{-1}
(I-A^{-\frac{1}{2}}BA^{-\frac{1}{2}}).
\end{align}
}
Multiplying both sides of (\ref{eq-1}) from  left and right by
$A^{\frac{1}{2}}$, yields
\begin{align*}
\nonumber &\big((1-\lambda)A+\lambda B\big)-\big((1-\lambda)A^{-1}+\lambda B^{-1}\big)^{-1}\\
&=\lambda(1-\lambda)(A^{\frac{1}{2}}-BA^{-\frac{1}{2}})
\Big((1-\lambda)A^{-\frac{1}{2}}BA^{-\frac{1}{2}}+\lambda I\Big)^{-1}
(A^{\frac{1}{2}}-A^{-\frac{1}{2}}B)\\
&=\lambda(1-\lambda)(A^{\frac{1}{2}}-BA^{-\frac{1}{2}})A^{\frac{1}{2}}A^{-\frac{1}{2}}
\Big((1-\lambda)A^{-\frac{1}{2}}BA^{-\frac{1}{2}}+\lambda I\Big)^{-1}
A^{-\frac{1}{2}}A^{\frac{1}{2}}(A^{\frac{1}{2}}-A^{-\frac{1}{2}}B)\\
&=\lambda(1-\lambda)(A-B)
A^{-\frac{1}{2}}\Big((1-\lambda)A^{-\frac{1}{2}}BA^{-\frac{1}{2}}+\lambda I\Big)^{-1}
A^{-\frac{1}{2}}
(A-B)\\
&=\lambda(1-\lambda)(A-B)\big((1-\lambda)B+\lambda A\big)^{-1}(A-B).
\end{align*}
So the proof of (i) is complete.\\
(ii) Similar to (i), putting
$T=A^{-\frac{1}{2}}BA^{-\frac{1}{2}}$ in (\ref{lem-f-s}),
we have
{\small\begin{align*}
&(A^{-\frac{1}{2}}BA^{-\frac{1}{2}})^{\frac{1}{2}}\Big[(1-\lambda)I+
\lambda A^{\frac{1}{2}}B^{-1}A^{\frac{1}{2}}-\big((1-\lambda)I+
\lambda A^{-\frac{1}{2}}BA^{-\frac{1}{2}} \big)^{-1}\Big]
(A^{-\frac{1}{2}}BA^{-\frac{1}{2}})^{\frac{1}{2}}\\
&=\lambda(1-\lambda)(A^{-\frac{1}{2}}BA^{-\frac{1}{2}}-I)
\big((1-\lambda)I+\lambda A^{-\frac{1}{2}}BA^{-\frac{1}{2}}\big)^{-1}
(A^{-\frac{1}{2}}BA^{-\frac{1}{2}}-I)
\end{align*}
}
or
{\small\begin{align}
\nonumber &(A^{-\frac{1}{2}}BA^{-\frac{1}{2}})^{\frac{1}{2}}A^{\frac{1}{2}}
\Big[\big((1-\lambda)A^{-1}+\lambda B^{-1}\big)-
\big((1-\lambda)A+\lambda B\big)^{-1}\Big]A^{\frac{1}{2}}
(A^{-\frac{1}{2}}BA^{-\frac{1}{2}})^{\frac{1}{2}}\\
\label{th-(ii)-p-1}&=\lambda(1-\lambda)A^{-\frac{1}{2}}(B-A)\big((1-\lambda)A+\lambda B\big)^{-1}(B-A)A^{-\frac{1}{2}}.
\end{align}
}
Multiplying both sides of (\ref{th-(ii)-p-1}) from  left and right by
$A^{\frac{1}{2}}$ we deduce (ii).\\
(iii) Putting $T=A^{-\frac{1}{2}}BA^{-\frac{1}{2}}$ in (\ref{lem-f-ex})
we get
{\small\begin{align*}
&\big((1-\lambda)I+\lambda A^{\frac{1}{2}}B^{-1}A^{\frac{1}{2}}\big)^{\frac{1}{2}}
\big((1-\lambda)I+\lambda A^{-\frac{1}{2}}BA^{-\frac{1}{2}})
\big((1-\lambda)I+\lambda A^{\frac{1}{2}}B^{-1}A^{\frac{1}{2}}\big)^{\frac{1}{2}}
-I\\
&=\lambda(1-\lambda)(I-A^{-\frac{1}{2}}BA^{-\frac{1}{2}})A^{\frac{1}{2}}B^{-1}A^{\frac{1}{2}}
(I-A^{-\frac{1}{2}}BA^{-\frac{1}{2}}).
\end{align*}
}
Therefore
{\small\begin{align*}
&\Big(A^{\frac{1}{2}}\big((1-\lambda)A^{-1}+\lambda B^{-1}\big)A^{\frac{1}{2}}\Big)^{\frac{1}{2}}
A^{-\frac{1}{2}}\big((1-\lambda)A+\lambda B\big)A^{-\frac{1}{2}}
\Big(A^{\frac{1}{2}}\big((1-\lambda)A^{-1}+\lambda B^{-1}\big)A^{\frac{1}{2}}\Big)^{\frac{1}{2}}
\\
&-I=\lambda(1-\lambda)A^{-\frac{1}{2}}(A-B)B^{-1}(A-B)A^{-\frac{1}{2}}.
\end{align*}
}
This is equivalent to
{\small\begin{align}
\nonumber&\Big(A^{\frac{1}{2}}\big(A!_\lambda B\big)^{-1}A^{\frac{1}{2}}\Big)^{\frac{1}{2}}
A^{-\frac{1}{2}}(A\nabla_\lambda B)A^{-\frac{1}{2}}
\Big(A^{\frac{1}{2}}\big(A!_\lambda B\big)^{-1}A^{\frac{1}{2}}\Big)^{\frac{1}{2}}-I\\
\label{th-(iii)-p-1}&=\lambda(1-\lambda)A^{-\frac{1}{2}}(A-B)B^{-1}(A-B)A^{-\frac{1}{2}}.
\end{align}
}
Multiplying both sides of (\ref{th-(iii)-p-1}) from  left and right by
$A^{\frac{1}{2}}$ we deduce (iii).
\end{proof}
The next theorem gives a natural operator version
of the Ky Fan type inequality (\ref{ky-fan-3}) in the case of $n=2$.
Operator extensions of (\ref{ky-fan-4}) and (\ref{ky-fan-5}), are also
presented. For the sake of brevity, we set $A':=I-A$ and $B':=I-B$.
\begin{theorem}
Let $A,B\in \mathbb{B}(\mathscr{H})$
such that $0<A,B\leq\frac{1}{2}I$ and $\lambda\in[0,1]$. Then
we have
{\small\begin{align}
\label{operator-ky-fan3}&(i)~~~A'\nabla_\lambda B'-A'!_\lambda B'
\leq A\nabla_\lambda B-A!_\lambda B.\\
\label{operator-ky-fan4}&(ii)~~~
(A'\sharp B')\Big((A'!_\lambda B')^{-1}-(A'\nabla_\lambda B')^{-1}\Big)
(A'\sharp B')\leq
(A\sharp B)\Big((A!_\lambda B)^{-1}-(A\nabla_\lambda B)^{-1}\Big)
(A\sharp B).\\
\nonumber&(iii)~~~
A'\Big(A'^{-1}\sharp(A'!_\lambda B')^{-1}\Big)(A'\nabla_\lambda B')
\Big(A'^{-1}\sharp(A'!_\lambda B')^{-1}\Big)A'-A'\\
\label{operator-ky-fan5}&\qquad\leq
A\Big(A^{-1}\sharp(A!_\lambda B)^{-1}\Big)(A\nabla_\lambda B)
\Big(A^{-1}\sharp(A!_\lambda B)^{-1}\Big)A-A.
\end{align}
}
\end{theorem}
\begin{proof}
(i) Since  $A',B'>0$, substituting $A$ and $B$ with $A'$ and $B'$
in (\ref{th-ex-lem-f}) respectively, we get
{\small\begin{equation}\label{cor-ex-i}
A'\nabla_\lambda B'-A'!_\lambda B'
=\lambda(1-\lambda)(A-B)(B'\nabla_\lambda A')^{-1}(A-B).
\end{equation}
}
Now since $0<A\leq A'$ and $0<B\leq B'$, we have
$$
0<B\nabla_\lambda A=(1-\lambda)B+\lambda A\leq (1-\lambda)B'+\lambda A'=B'\nabla_\lambda A',
$$
and so
$$
\big(B'\nabla_\lambda A'\big)^{-1}\leq\big(B\nabla_\lambda A\big)^{-1}.
$$
Hence
$$
(A-B)\big(B'\nabla_\lambda A'\big)^{-1}(A-B)
\leq
(A-B)\big(B\nabla_\lambda A\big)^{-1}(A-B).
$$
Now considering (\ref{th-ex-lem-f}) and (\ref{cor-ex-i})
we obtain (\ref{operator-ky-fan3}).
The proofs of (ii) and (iii) are similar to (i) and we omit the details.
\end{proof}
\begin{remark}
(i) In the process of proving  (\ref{th-ex-lem-f-s}) and
(\ref{th-ex-lem-f-ex}), we used  (\ref{lem-f-s}) and (\ref{lem-f-ex}).
If in (\ref{p-l-(ii)}), instead of  multiplying $T^{\frac{1}{2}}$ from  left and right, 
 multiplying $T$ only from right or only from left respectively, and then substituting
$T=A^{-\frac{1}{2}}BA^{-\frac{1}{2}}$ and multiplying $A^{\frac{1}{2}}$ from left and right,
we obtain the following chain of identities 
{\small\begin{align}
\nonumber&~(A\sharp B)\Big[(A!_\lambda B)^{-1}-(A\nabla_\lambda B)^{-1}\Big]
(A\sharp B)=A\Big[(A!_\lambda B)^{-1}-(A\nabla_\lambda B)^{-1}\Big]B\\
\label{eq1}&=B\Big[(A!_\lambda B)^{-1}-(A\nabla_\lambda B)^{-1}\Big]A
=\lambda(1-\lambda)(B-A)(A\nabla_\lambda B)^{-1}(B-A).\qquad\qquad
\end{align}
}
Similarly related to (\ref{th-ex-lem-f-ex}), instead of multiplying
(\ref{lem-f}) from left and right by  $(I!_\lambda T)^{-\frac{1}{2}}$, multiplying
(\ref{lem-f}) only form right or only from left by
$(I!_\lambda T)^{-1}$ and continuing in the same manner, 
yield 
{\small\begin{align}
\nonumber&A\Big(A^{-1}\sharp(A!_\lambda B)^{-1}\Big)(A\nabla_\lambda B)
\Big(A^{-1}\sharp(A!_\lambda B)^{-1}\Big)A=A(A!_\lambda B)^{-1}(A\nabla_\lambda B)\qquad\qquad\qquad\\
\label{eq2}&=(A\nabla_\lambda B)(A!_\lambda B)^{-1}A
=\lambda(1-\lambda)(A-B)B^{-1}(A-B)+A.
\end{align}
}
\noindent (ii) In the case that $A,B$  are two commutative operators
such that $0<A,B\leq\frac{1}{2}I$ and $\lambda\in[0,1]$,
the identites (\ref{th-ex-lem-f-s}) and (\ref{th-ex-lem-f-ex})
transform to
{\small\begin{align}
\label{rem-comm-ii}&(A!_\lambda B)^{-1}-(A\nabla_\lambda B)^{-1}
=\lambda(1-\lambda)(B-A)(A\sharp B)^{-2}(A\nabla_\lambda B)^{-1}(B-A)\\
\label{rem-comm-iii}& (A!_\lambda B)^{-1}(A\nabla_\lambda B)-I
=\lambda(1-\lambda)(A-B)A^{-1}B^{-1}(A-B).
\end{align}
}
Now changing $A,B$ by $A',B'$ we obtain
{\small\begin{align}
\label{rem-comm-ii-'}&(A'!_\lambda B')^{-1}-(A'\nabla_\lambda B')^{-1}
=\lambda(1-\lambda)(B-A)(A'\sharp B')^{-2}(A'\nabla_\lambda B')^{-1}(B-A)\\
\label{rem-comm-iii-'}& (A'!_\lambda B')^{-1}(A'\nabla_\lambda B')-I
=\lambda(1-\lambda)(A-B)A'^{-1}B'^{-1}(A-B).
\end{align}
}
Comparing (\ref{rem-comm-ii}) and (\ref{rem-comm-ii-'}), yields
{\small\begin{equation}\label{ii-ii'}
(A'!_\lambda B')^{-1}-(A'\nabla_\lambda B')^{-1}\leq (A!_\lambda B)^{-1}-(A\nabla_\lambda B)^{-1}.
\end{equation}
}
Similarly, comparing (\ref{rem-comm-iii}) and (\ref{rem-comm-iii-'})
consequences
{\small\begin{equation}\label{iii-iii'}
(A'!_\lambda B')^{-1}(A'\nabla_\lambda B')\leq (A!_\lambda B)^{-1}(A\nabla_\lambda B).
\end{equation}
}
\noindent If $x_1,x_2\in(0,\frac{1}{2}]$ and $\lambda\in[0,1]$,
 putting $A=x_1I$ and $B=x_2I$ in
(\ref{operator-ky-fan3}), (\ref{ii-ii'}) and (\ref{iii-iii'}),
we get (\ref{ky-fan-3}), (\ref{ky-fan-4}) and (\ref{ky-fan-5}) when $n=2$.
These show that (\ref{operator-ky-fan3}), (\ref{operator-ky-fan4}) and (\ref{operator-ky-fan5})
are operator extensions of (\ref{ky-fan-3}), (\ref{ky-fan-4}) and (\ref{ky-fan-5}).

\end{remark}

 \end{document}